\documentclass[12pt,reqno,twoside]{amsart}
\usepackage{amsthm,amsmath,amssymb,mathrsfs,amsbsy}
\usepackage{graphicx,epsfig,wrapfig,sidecap,subcaption}
\usepackage[all]{xy}
\usepackage{proof}
\usepackage[usenames]{color}
\usepackage{hyperref}
\usepackage{bm}
\usepackage[capitalize]{cleveref}
\usepackage{stmaryrd}
\usepackage{tikz}
\usepackage{mathdots}

\xyoption{dvips}
\xyoption{color}
\usepackage{a4wide} 	

\usepackage{color}

\title[Visibility Properties of Spiral Sets]{Visibility Properties of Spiral Sets}
\author{Faustin Adiceam and Ioannis Tsokanos}
\date{}

\newcommand{\Z}{{\mathbb{Z}}}

\newcommand{\R}{{\mathbb{R}}}

\newcommand{\Sph}{\mathbb{S}}

\newcommand{\dist}{\textrm{dist}}

\theoremstyle{plain}
\newtheorem{thm}{Theorem}[section]

\newtheorem{prop}[thm]{Proposition}

\theoremstyle{definition}
\newtheorem{definition}[thm]{Definition}

\numberwithin{equation}{section}
\swapnumbers

\usepackage{color}

\usepackage{tikz-3dplot}
\tdplotsetmaincoords{80}{110}

\begin{document}

\begin{abstract}
A spiral in $\R^{d+1}$ is defined as a set of the form $\left\{\sqrt[d+1]{n}\cdot\bm{u}_n\right\}_{n\ge 1},$ where $\left(\bm{u}_n\right)_{n\ge 1}$ is a spherical sequence. Such point sets have been extensively studied, in particular in the planar case $d=1$, as they then serve as natural models describing phyllotactic structures (i.e.~structures representing configurations of leaves on a plant stem).

Recent progress in this theory provides a fine analysis of the distribution of spirals (e.g., their covering and packing radii). Here, various concepts of visiblity from discrete geometry are employed to characterise density properties of such point sets. More precisely, necessary an sufficient conditions are established for a spiral to be (1) an \emph{orchard} (a ``homogeneous'' density property defined by P\`olya), (2) a \emph{uniform orchard} (a concept introduced in this work), (3) a \emph{set with no visible point} (implying that the point set is dense enough in a suitable sense) and (4) a \emph{dense forest} (a quantitative and uniform refinement of the previous concept).
\end{abstract}

\maketitle

\begin{center}
\emph{\`A Assumpta Adic\'eam.}
\end{center}



\section{Introduction}\label{sec:introduction} 

Define a spiral in $\R^{d+1}$ as a set of points of the form 
\begin{equation}\label{defspiral}
\mathfrak{S}_d\left(\bm{U}\right)\;=\; \left\{\sqrt[d+1]{n}\cdot\bm{u}_n\right\}_{n\ge 1},
\end{equation}
where $\bm{U}=\left(\bm{u}_n\right)_{n\ge 1}$ is a sequence in the $d$--dimensional unit sphere $\Sph^d\subset\R^{d+1}$. Such discrete point sets appear naturally in the modelisation of phyllotactic models (i.e.~in the modelisation of configurations of leaves on a plant stem) and have long been studied in this respect. Recently, Akiyama~\cite{akiyama} initiated a systematic study of their distribution  properties. In the planar case $d=1$,  he thus established that the spiral obtained from the circle sequence $\bm{U}=\left(\exp\left(2i\pi n\theta\right)\right)_{n\ge 1}$ is \emph{Delone} if and ony if $\theta$ is a badly approximable number; that is, if and only if 
\begin{equation}\label{defbad}
\inf_{q\ge 1, r\in\Z} q\left|q\theta -r\right|>0.
\end{equation} 
Recall that a set is Delone in $\R^{d+1}$ if it is both uniformly discrete (the infimum of the distances between two distinct points in the set is positive) and relatively dense (the supremum of the distances from any point in $\R^{d+1}$ to the set is finite). Marklof~\cite{marklof} then generalised this result by providing a necessary and sufficient condition for a given planar spiral to be Delone. \\

The resulting question of determining whether this theory can be extended to higher dimensions was later settled by the authors~: they indeed established in~\cite{aditsok} a necessary and sufficient condition for a spiral set to be Delone in $\R^{d+1}$ and described an explicit construction proving that such point sets do exists in any dimension. In particular, given a sequence $\left(\bm{u}_n\right)_{n\ge 1}$ in $\Sph^d$ and an increasing function $f~:\R_+\rightarrow\R_{+}$, it is noticed in~\cite[\S 1]{aditsok} that a necessary condition for a point set of the form $\left\{f(n)\cdot \bm{u}_n\right\}_{n\ge 1}$ to be uniformy discrete and relatively dense is that $0<\liminf_{n\rightarrow\infty}f(n)/\sqrt[d+1]{n}$ and $\limsup_{n\rightarrow\infty} f(n)/\sqrt[d+1]{n}<\infty$, respectively. This motivates the choice of the factor $\sqrt[d+1]{n}$ in the definition of a spiral set in~\eqref{defspiral}.\\

The goal of this paper is to study distribution properties of spiral sets from a complementary standpoint; namely, from that of so--called visibility problems in discrete geometry which quantify in suitable senses the density of a point set. More precisely, the goal is to establish necessary and sufficient conditions for a spiral set to become arbitrarily close to line segments, provided they are long enough. This can be formalised in several distinct ways.\\

A first possible formalisation is motivated by P\'olya's classical orchard's problem stated in~\cite[Problem 239]{polya}~: ``how thick  must  the trunks of the trees in a regularly spaced circular orchard grow if they are to block completely the view from the center?''. Assume that the observer stands at the origin, that  the centres of the trees are located at the nonzero points of the lattice $\Z^2$ and that they have  Euclidean norms at most $Q>0$. P\'olya (\emph{ibid.}), based on a method due to Speiser, and then Honsberger~\cite{honsb}, based on Mikowski's Convex Body Theorem, showed that it is enough for the trees to have radius $1/Q$ to block any ray emerging from the origin. Slightly sharper estimates are also obtained by Allen~\cite{allen}, and some generalisations studied in~\cite[\S 4]{adiarith}.\\

The following definition introduces the concept of an orchard in full generality. Before stating it, recall that a set $\mathcal{B}\subset\R^{d+1}$ has \emph{finite density} if $$\limsup_{T\rightarrow \infty}\frac{\#\left(B_2\left(\bm{0}, T\right)\cap\mathcal{B}\right)}{T^{d+1}}\;<\;\infty,$$ where $\#$ denotes the cardinality function and where, given $\bm{x}\in\R^{d+1}$ and $T>0$, $B_2\left(\bm{x}, T\right)$ stands for the Euclidean ball centered at $\bm{x}$ with radius $T>0$. In visibility problems, it is natural to work under the assumption that point sets have finite density for  the visibility conditions under consideration not to be trivially met. It is easy to see that spiral sets defined in~\eqref{defspiral} have finite density.

\begin{definition}[Orchard] \label{orchard}
A subset $\mathcal{O}\subset \R^{d+1}\backslash\left\{\bm{0}\right\}$ is an \emph{orchard} if it has  finite density and if there exists a function $V~: \epsilon\in(0,1)\rightarrow V(\epsilon)\in\R_+$, increasing as $\epsilon\rightarrow 0^+$, such that the following holds~:  for every $\epsilon>0$ and every direction $\bm{v}\in\Sph^d$, there exists a point $\bm{o}\in\mathcal{O}$ and a real number $0<t<V(\epsilon)$ such that $\left\|\bm{o} - t\bm{v}\right\|_2\le \epsilon$.
\end{definition}


As is not hard to see, a visibility function $V$ in an orchard in dimension $(d+1)$ has to satisfy the bound 
\begin{equation}\label{boundvisibility}
V(\epsilon)\ge c\cdot \epsilon^{-d}
\end{equation} 
for some constant $c>0$. This will be justified in detail in Section~\ref{2}.\\

In Definition~\ref{orchard}, a given point set is required to become arbitrary close to long enough line segments which have the origin as one of their end points. Removing this assumption while keeping the constraint that the line segments must be supported on  directions passing through the origin leads one to the concept of a \emph{uniform orchard}, which is introduced in this paper. 

\begin{definition}[Uniform Orchard] \label{ouv}
A subset $\mathcal{O}\subset \R^{d+1}\backslash\left\{\bm{0}\right\}$ is a \emph{uniform orchard} if it has  finite density and if there exists a function $V~: \epsilon\in(0,1)\rightarrow V(\epsilon)\in\R_+$, increasing as $\epsilon\rightarrow 0^+$, such that the following holds~:  for every $\epsilon>0$, every $t_0\in \R$ and every direction $\bm{v}\in\Sph^d$, there exists a point $\bm{o}\in\mathcal{O}$ and a real number $t_0<t<t_0+V(\epsilon)$ such that $\left\|\bm{o}- t\bm{v}\right\|_2\le \epsilon$.
\end{definition}

Clearly, a uniform orchard is also an orchard (just take $t_0=0$ in the above definition). The converse, however, does not hold and, despite the apparent similarity in their definitions, these two concepts should be thought of as being rather different in nature~: indeed, Section~\ref{2} provides examples of an orchard and of a uniform orchard which have drastically different visibility properties (in a sense made precise therein).\\

Removing the assumption in Definitions~\ref{orchard} and~\ref{ouv} that the line segments should be supported on directions passing through the origin leads one to  further concepts of visibility which have previously appeared in the literature --- see, e.g., \cite{bossol}. In order to state them, define the \emph{ray emanating from a point $\bm{x}\in\R^{d+1}$ in direction $\bm{v}\in\Sph^d$} as the half--line 
\begin{equation}\label{defhalfline}
L(\bm{x}, \bm{v})\;=\; \left\{\bm{x}+t\bm{v}\right\}_{t\ge 0}.
\end{equation} 
Also, given two subsets $\mathcal{A}, \mathcal{B}\subset\R^{d+1}$, set $$\dist_2(\mathcal{A}, \mathcal{B})\; =\; \inf_{\bm{a}\in \mathcal{A}, \bm{b}\in \mathcal{B}}\left\|\bm{a}-\bm{b}\right\|_2.$$

\begin{definition}[Visible Points] Let $\mathfrak{Y}\subset\R^{d+1}$ be non--empty. Fix 
a direction $\bm{v}\in\Sph^d$.

The set of \emph{visible points of $\mathfrak{Y}$ in direction $\bm{v}$} is defined as $$\textrm{Vis}\left(\mathfrak{Y}, \bm{v}\right)\;=\; \left\{\bm{x}\in\R^{d+1}\: :\: \dist_2\left(L(\bm{x}, \bm{v}), \,\mathfrak{Y}\backslash\left\{\bm{x}\right\} \right)>0\right\}.$$ 

The set of \emph{visible points of $\mathfrak{Y}$} is defined as $$\textrm{Vis}\left(\mathfrak{Y}\right)\;=\; \left\{\bm{x}\in\R^{d+1}\: :\: \exists \bm{v}\in\Sph^d, \; \bm{x}\in \textrm{Vis}\left(\mathfrak{Y}, \bm{v}\right)\right\}.$$

The set of \emph{hidden points of $\mathfrak{Y}$} is the complement $\R^{d+1}\backslash \textrm{Vis}\left(\mathfrak{Y}\right)$.
\end{definition}

Making the concept of $\epsilon$--hidden points more quantitative leads one to the concept of a \emph{dense forest}.

\begin{definition}[Dense Forest]
A subset $\mathcal{F}\subset \R^{d+1}$ is an \emph{dense forest} if it has  finite density and if there exists a function $V~: \epsilon\in(0,1)\rightarrow V(\epsilon)\in\R_+$, increasing as $\epsilon\rightarrow 0^+$, such that the following holds~:  for every reals $\epsilon>0$ and $t_0\in \R$, every vector $\bm{x}\in\R^{d+1}$ and every direction $\bm{v}\in\Sph^d$, there exists a point $\bm{f}\in\mathcal{F}$ and a real number $t_0<t<t_0+V(\epsilon)$ such that $\left\| \bm{x}+t\bm{v}-\bm{f}\right\|_2\le \epsilon$.
\end{definition}

Clearly, a dense forest is both a set whose hidden point set is the full space and a uniform orchard (and therefore also an orchard). The latter implies in particular that the lower bound~\eqref{boundvisibility} still holds for the visibility function in a dense forest.\\

The concept of a dense forest emerged in~\cite{Bishop} in a question of rectifiability of curves and was later substantially developed in relation with the Danzer Problem, e.g.~in~\cite{Alon, Adiceam, ASW, SW, tsok}. The Danzer Problem asks for the existence of a set with  finite density intersecting any convex body of volume one. Relaxing the assumption on the volume of the convex bodies gives rise to the construction of dense forests in this context~: as detailed in~\cite{adidanz}, a set solution to the Danzer problem in $\R^{d+1}$ is a dense forest with visibility $V(\epsilon)=C\cdot \epsilon^{-d}$ for some $C>0$ (which is, up to the multiplicative constant, optimal in view of~\eqref{boundvisibility}), and the converse holds true in the planar case $d=1$. \\

Some of the visibility properties of the spiral $\mathfrak{S}_d\left(\bm{U}\right)$ in~\eqref{defspiral} will be expressed in terms of a covering property of the  spherical sequence $\bm{U}$  introduced in the following definition. To this end, the natural geodesic metric on $\Sph^d$ is denoted by $\dist_{\Sph^d}\left(\,\cdot\, ,\, \cdot \,\right)$.

\begin{definition}[Uniform Covering Parameter of a Spherical Sequence] Let $\bm{U}=\left(\bm{u}_n\right)_{n\ge 1}$ be an infinite sequence in $\Sph^d$. Given a constant $C>0$, define its \emph{uniform covering parameter at scale $C>0$} as the quantity $$R_d(\bm{U}, C)\;=\; \sup_{\bm{v}\in\Sph^d}\;\;\;\sup_{m\ge 0}\;\;\;\sup_{N\ge 1}\;\;\; \sqrt[d]{N}\cdot\left(\inf_{1\le n\le C\cdot N}\dist_{\Sph^d}\left(\bm{u}_{m+n} ,\bm{v}\right)\right).$$

In the case that the sequence $\bm{U}$ is finite, say $\bm{U}=\left(\bm{u}_n\right)_{1\le n\le N}$ for some integer $N\ge 1$, set $$R_d(\bm{U})\;=\; \sup_{\bm{v}\in\Sph^d}\;\min_{1\le n\le N}\; \dist_{\Sph^d}\left(\bm{u}_{n} ,\bm{v}\right).$$
\end{definition}

The main result of this paper is then the following~:

\begin{thm}\label{main}
Let $\mathfrak{S}_d\left(\bm{U}\right)$ be a spiral point set in $\R^{d+1}$ defined   from a spherical sequence $\bm{U}=\left(\bm{u}_n\right)_{n\ge 1}$ as in~\eqref{defspiral}. Let $V~: \epsilon\in(0,1)\rightarrow V(\epsilon)\in\R_+$ be a function increasing  as $\epsilon\rightarrow 0^+$. Given constants $c_{\bm{U}}\ge 1$ and $\kappa_{\bm{U}}>0$ depending only on $\bm{U}$, define the auxiliary function 
\begin{equation}\label{defW}
W~:\epsilon\mapsto c_{\bm{U}}\cdot V(\kappa_{\bm{U}}\cdot\epsilon).
\end{equation} 

Given a real $x\ge 0$ and an integer $h\ge 1$, denote by $A_{\bm{U}}\left(h, x\right)$ the finite set 
\begin{equation}\label{defAuhx}
A_{\bm{U}}\left(h, x\right)\;=\; \left\{\bm{u}_{h^{d+1}+i}\; :\; 1\,\le\, i \,\le\, x \right\}\;\subset\; \Sph^d.
\end{equation}
\\
\begin{enumerate}
\item \textbf{(Spirals and orchards)} The spiral $\mathfrak{S}_d\left(\bm{U}\right)$ forms an orchard with visibility $W$ for some constant $c_{\bm{U}}\ge 1$ and $\kappa_{\bm{U}}>0$ if and only if the following holds  for some constants $\kappa, K>0$~: for all $\epsilon>0$ small enough and  all $\bm{v}\in\Sph^d$, there exists an integer $n$ such that $1\le n\le K\cdot V(\epsilon)^{d+1}$ and $\dist_{\Sph^d}\left(\bm{u}_n, \bm{v}\right)\le \kappa\cdot\epsilon/\sqrt[d+1]{n}$\label{pt1}.\\

\item \textbf{(Spirals and uniform orchards)} 
The following two claims are equivalent~: \label{pt2}\\
\begin{enumerate}
\item  There exist constants $c_{\bm{U}}\ge 1$ and $\kappa_{\bm{U}}>0$ such that the spiral $\mathfrak{S}_d\left(\bm{U}\right)$ is a uniform orchard with visibility function $W$.\label{pt2a}\\

\item  There exist constants $K, \kappa_{\bm{U}}>0$ and $c_{\bm{U}}\ge 1$ such that  $$\sup_{\epsilon\in (0,1)}\;\sup_{h\ge W\left(\epsilon\right)}\; h\cdot \epsilon^{-1}\cdot R_d\left(A_{\bm{U}}\left(h,\,K\cdot h^d \cdot W\left(\epsilon\right)\right)\right)\; <\; \infty, $$ where the function $W$ is defined as in~\eqref{defW}. \label{pt2b}\\
\end{enumerate}

Furthermore~:\\
\begin{itemize}
\item[(i)] whenever $R_d\left(\bm{U}, C\right)<\infty$ for some $C>0$, the spiral $\mathfrak{S}_d\left(\bm{U}\right)$ is a uniform orchard with optimal visibility $C_{\bm{U}}\cdot \epsilon^{-d}$ for some constant $C_{\bm{U}}>0$.\label{2i}\\

\item[(ii)] a sufficient condition for the set $\mathfrak{S}_d\left(\bm{U}\right)$ to be a uniform orchard is that the relation 
\begin{equation*}\label{condi3}
\lim_{x\rightarrow\infty}\left(\sup_{h> x} \,h\cdot R_d\left(A_{\bm{U}}\left(h, h^dx\right)\right)\right)\;=\;0
\end{equation*}
should hold. In this case, there exist constants $c_{\bm{U}}\ge 1$ and $\kappa_{\bm{U}}>0$ such that it admits the function $W$ defined in~\eqref{defW} as a visibility function obtained from the map
\begin{equation}\label{defvisi}
V~:\epsilon\mapsto \sup\left\{x\ge 0\; :\; \sup_{h> x} \,h\cdot R_d\left(A_{\bm{U}}\left(h, h^dx\right)\right)\ge \epsilon\right\}.
\end{equation}
\end{itemize}

\item \textbf{(Visible points in spirals)} The following two claims hold~: \label{pt3}\\
\begin{itemize}
\item[(i)] The spiral $\mathfrak{S}_d\left(\bm{U}\right)$ has an empty set of visible points if and only if the following holds for some  constant $c>0$~:  for any $\epsilon>0$, any $\lambda\ge 0$ and any choice of orthogonal  directions $\bm{v}, \bm{w}\in\Sph^d$, there exist a real number $t$ and an integer $n\ge 1$  such that 
\begin{equation}\label{eqvisi}
\left|\sqrt[d+1]{n}-\sqrt{\lambda^2+t^2}\right|\le \epsilon  \qquad \textrm{and} \qquad \dist_{\Sph^d}\left(\bm{u}_n, \frac{\lambda\bm{v}+t\bm{w}}{\sqrt{\lambda^2+t^2}}\right)\le \frac{c\cdot\epsilon}{\sqrt{\lambda^2+t^2}}.\\
\end{equation} 
are satisfied.\\
\item[(ii)]  A spiral set which is a uniform orchard has an empty set of visible points (in other words,  all points in $\R^{d+1}$ are hidden). The converse, however, does not hold. \label{pt3}\\
\end{itemize}

\item \textbf{(Spirals and dense forests)}  \label{pt4} The spiral $\mathfrak{S}_d\left(\bm{U}\right)$ is a dense forest with visibility function $W$ for some constants $c_{\bm{U}}\ge 1$ and $\kappa_{\bm{U}}>0$ if and only if there exists a constant $c>0$ such that for any $\epsilon>0$, $\lambda\ge 0$, $t_0\in\R$ and any choice of orthogonal  directions $\bm{v}, \bm{w}\in\Sph^d$, there exist  $t\in\left[t_0, t_0+V\left(\epsilon\right)\right]$ and an integer $n\ge 1$ such that the inequalities~\eqref{eqvisi} are simultaneously met.\\ 
\end{enumerate}

\end{thm}

Note that, if the function $V$ has polynomial growth in the sense that there exists $\alpha>0$ such that $V\left(\epsilon\right)\le c\cdot \epsilon^{-\alpha}$ for some $c>0$, the constant $\kappa_{\bm{U}}>0$ appearing in the definition of the function $W$ in~\eqref{defW} can be absorbed in the constant $c_{\bm{U}}$.\\


Explicit spherical sequences with finite uniform covering parameters are constructed in any dimension  by the authors in~\cite[\S\S 3 \& 4]{aditsok}.  As a consequence of Points~\ref{2i} and~3(ii), this establishes the existence of spiral sets which are uniform orchards and have an empty set of visible points. Furthermore, the construction provided in~\cite[\S\S 3 \& 4]{aditsok} ensures that the resulting point sets enjoy the additional property of being Delone. An example in the  plane of such spirals is the Fermat (or sunflower) spiral 
\begin{equation}\label{equasunf}
\left\{\sqrt{n}\cdot\exp\left(2i\pi n\varphi\right)\right\}_{n\ge 1},\qquad  \textrm{where} \qquad  \varphi=(1+\sqrt{5})/2
\end{equation} 
is the Golden Ratio. It is represented in Figure~\ref{figspir}.\\

\begin{figure}[h!]
\centering
\begin{subfigure}{.5\textwidth}
  \centering
  \includegraphics[scale=0.6]{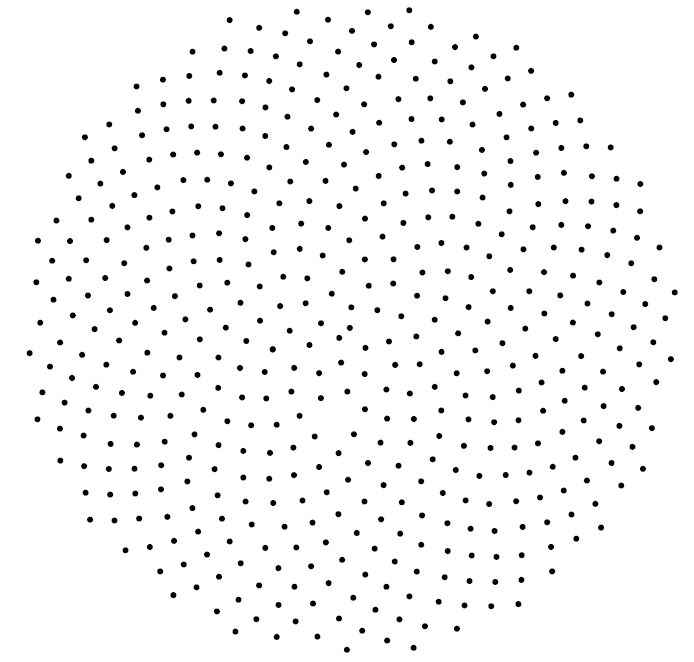}
\end{subfigure}%
\begin{subfigure}{.5\textwidth}
  \centering
  \includegraphics[scale=0.6]{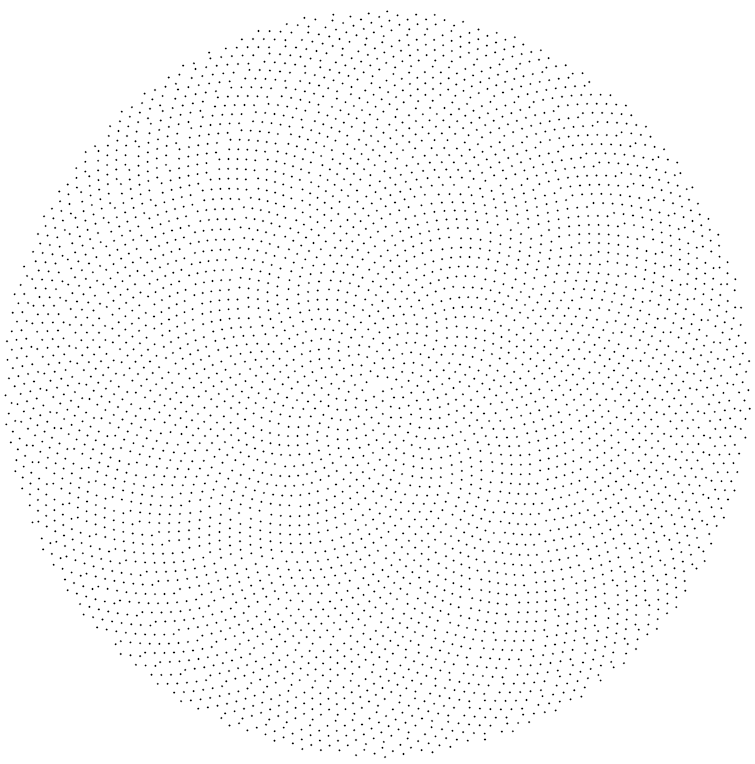}
\end{subfigure}
\caption{The sunflower spiral~\eqref{equasunf} at two different scales.}
\label{figspir}
\end{figure}

By contrast, the existence of a spiral set which is a dense forest remains an open question. As a matter of fact, the authors conjecture that there is no such set but have not been able to establish this claim. The case of the sunflower spiral 
is already elusive as specialising Point~\ref{pt4} to this situation leads one to a new kind of moving and shrinking target problem modulo one. This particular example is all the more intriguing as it has recently been established that the areas of the Voronoi cells of the spiral under consideration converge in an appropriate sense, therefore conferring to the resulting point set an additional regularity property. See~\cite{yamavor} for definitions and proofs.\\

Theorem~\ref{main} is established in Section~\ref{2}. Throughout, the Vinogradov symbol will be used to simplify notation~: given two positive real quantities $x$ and $y$, the existence of a constant $c>0$ independent of $x$ and $y$, referred to as the implicit constant, will be denoted by $x\ll y$. Similarly, $x\gg y$ means that $x^{-1}\ll y^{-1}$ and $x\asymp y$ that the relations $x\ll y$ and $x\gg y$ hold simultaneously.

\section{Proof of the Main Theorem}\label{2}

The claim on the lower bound~\eqref{boundvisibility}  for a visibility function in an orchard is first established. The proof rests on the following relations, which are readily verified~: given reals $\rho, r>0$ and unit vectors $\bm{u}, \bm{v}\in\Sph^d$, 
\begin{equation}\label{identity}
\dist_{\Sph^d}\left(\bm{u},\bm{v}\right)\;\asymp\;\left\|\bm{u}-\bm{v}\right\|_2\qquad \textrm{and }\qquad \left\|r\bm{u}-\rho\bm{v}\right\|_2\; \asymp\; \left|r-\rho\right|+\sqrt{r \rho}\left\|\bm{u}-\bm{v}\right\|_2
\end{equation}
(working in the plane spanned by $\bm{u}$ and $\bm{v}$, the second relation is easily deduced from the usual formula for the Euclidean distance between two points expressed in polar coordinates.)

\begin{proof}[Proof of the lower bound~\eqref{boundvisibility}] Let $\mathcal{O}\subset\R^{d+1}\backslash\left\{\bm{0}\right\}$ be an orchard with visibility function $V$. Enumerate the points in $\mathcal{O}$ in a set $\left\{\bm{x}_k=\rho_k\bm{u}_k\right\}_{k\ge 1}$ in such a way that $\left\|\bm{x}_k\right\|_2=\rho_k\le\left\|\bm{x}_{k+1}\right\|_2=\rho_{k+1}$ for all $k\ge 1$ (in particular, $\bm{u}_k\in\Sph^d$ for all $k\ge 1$). Since $\mathcal{O}$ has finite density, 
\begin{equation}\label{lowbrhok}
\rho_k\gg \sqrt[d+1]{k}.
\end{equation}

By assumption, given a direction $\bm{v}\in\Sph^d$, there exists a real number 
\begin{equation}\label{rhovsi}
\rho\in \left(0, V\left(\epsilon\right)\right)
\end{equation} 
and an index $k\ge 1$ such that $\epsilon\;\ge\; \left\|\rho\bm{v}-\rho_k\bm{u}_k\right\|_2$. 
From~\eqref{identity}, this implies that $\rho=\rho_k+\kappa\cdot\epsilon$ for some $\kappa\in\R$. If $\epsilon \ll \min_{j\ge 1}\rho_j$ (which may be assumed without loss of generality), then 
\begin{equation}\label{asymprhokrho}
\rho_k\asymp\rho.
\end{equation} 
Consequently, again from~\eqref{identity}, $\dist_{\Sph^d}\left(\bm{u}_k,\bm{v}\right)\;\asymp\;\left\|\bm{u}_k-\bm{v}\right\|_2\ll \epsilon/\rho_k.$

As this holds for any direction $\bm{v}\in\Sph^d$, it follows that the successive spherical caps centered at $\bm{u}_k$ with radii $C\cdot\epsilon/\rho_k$, where $1\le k\ll V(\epsilon)^{d+1}$ and where $C>0$ is a constant independent of $k$, cover the sphere $\Sph^d$. Here, the upper bound for the index $k$ is obtained by combining relations~\eqref{lowbrhok}, \eqref{rhovsi} and\eqref{asymprhokrho}.

Denoting by $A_d$ the surface area of  $\Sph^d$, one thus deduces that, for some constant $C>0$, $$A_d\;\ll\; \sum_{k=1}^{C\cdot V\left(\epsilon\right)^{d+1}}\frac{\epsilon^d}{\rho_k^d}\;\ll\; \sum_{k=1}^{V\left(\epsilon\right)^{d+1}}\frac{\epsilon^d}{k^{d/(d+1)}}\;\asymp\; \epsilon^d\cdot V\left(\epsilon\right).$$ The claim follows.
\end{proof}

The direct implication in Point~3(ii) in Theorem~\ref{main} is now proved seperately from the other points. This is because it holds in a generality greater than that of spiral point sets as detailed in the following proposition.

\begin{prop}\label{propunorvis}
Let $\mathcal{O}\subset\R^{d+1}\backslash\left\{\bm{0}\right\}$ be a uniform orchard. Then, its set of visible points is empty.
\end{prop}

\begin{proof}
Let $V$ denote a visibility function for the uniform orchard $\mathcal{O}$. Fix a vector $\bm{x}\in\R^{d+1}$, a direction $\bm{v}\in\Sph^d$ and a real $\epsilon>0$. The goal is to show that $\dist_2\left(L(\bm{x}, \bm{v}), \mathcal{O}\right)\le \epsilon$, where $L(\bm{x}, \bm{v})$ is the ray emanating from $\bm{x}$ in direction $\bm{v}$ as defined in~\eqref{defhalfline}. 

Let $\bm{v}'\in\Sph^d$ be a direction such that $L(\bm{x}, \bm{v})\cap L(\bm{0}, \bm{v}')\neq\emptyset$ and 
\begin{equation}\label{dotprod}
\left\|\bm{v}-\bm{v}'\right\|_2=\epsilon\cdot\left(4\pi\cdot V\left(\epsilon/2\right)\right)^{-1}.
\end{equation}
 Set $\left\{\bm{w}\right\}=L(\bm{x}, \bm{v})\cap L(\bm{0}, \bm{v}')$ and let $t_0\ge 0$ be such that $\bm{w}=t_0\bm{v}'$. Since $\mathcal{O}$ is a uniform orchard, there exists 
\begin{equation}\label{boundt}
0\le t\le V\left(\epsilon/2\right)
\end{equation} 
and $\bm{y}\in\mathcal{O}$ such that 
\begin{equation}\label{distmeas}
\left\|(t_0+t)\cdot\bm{v}'-\bm{y}\right\|_2\;=\; \left\|\bm{w}+t\cdot\bm{v}'-\bm{y}\right\|_2\;\le\;\frac{\epsilon}{2}.
\end{equation}

Consequently, 
\begin{align*}
\dist_2\left(L(\bm{x}, \bm{v}), \mathcal{O}\right)\;&\le\; \left\|\left(\bm{w}+t\cdot\bm{v}\right)-\bm{y}\right\|_2\\
&\le\; \left\|\bm{w}+t\cdot\bm{v}'-\bm{y}\right\|_2+t\cdot\left\|\bm{v}-\bm{v}'\right\|_2\\
&\underset{\eqref{dotprod}, \eqref{boundt}\&\eqref{distmeas}}{\le}\; \epsilon.
\end{align*}
This completes the proof.
\end{proof}

Proposition~\ref{propunorvis} justifies the claim made in the introduction, namely that the concepts of an orchard on the one hand and that of a uniform orchard on the other are rather different in nature when considering their visibility properties. This can already be seen in the case of spiral sets, which can be both orchards and have a non--empty set of visible points (compare with the statement of Proposition~\ref{propunorvis}). As established in the following statement, an example of such spiral in the plane is obtained by considering the point set $$\Lambda=\left\{\sqrt{n}\cdot \exp\left(2i\pi\cdot p/k\right)\right\}_{n\ge 1}.$$ Here, given an integer $n\ge 1$, the integers $k$ and $p$ are defined by the unique decomposition of $n$ as 
\begin{equation}\label{defn}
n=\frac{k(k+1)}{2}+p\qquad \textrm{with} \qquad k\ge 1\qquad  \textrm{and} \qquad 0\le p\le k. 
\end{equation}

\begin{prop}
The spiral $\Lambda$ is an orchard with visibility $V(\epsilon)\ll\epsilon^{-1}$ but has a non--empty set of visible points.
\end{prop}

The spiral $\Lambda$ is depicted in Figure~\ref{FigDirichletSpiral} below. In the same way as for the standard lattice in P\'olya's original orchard problem, the points close to the origin play a preponderant role in blocking rays emanating from the origin. By constrast, horizontal half--lines not passing through the origin but close to it  determine visible points.

\begin{figure}[h!]
	\centering
	\includegraphics[scale=0.75]{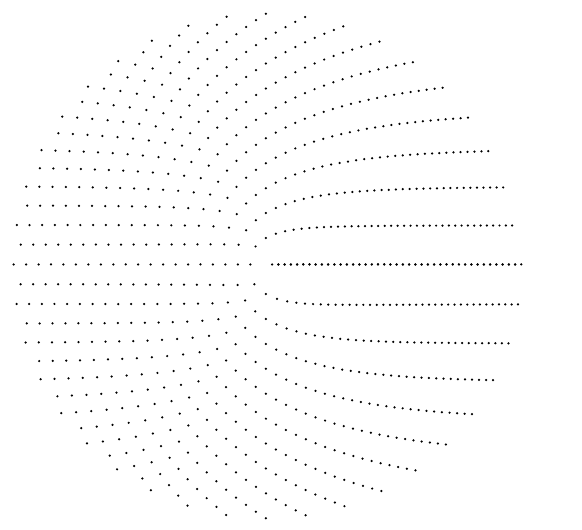}
	\caption{The Spiral $\Lambda$.}
	\label{FigDirichletSpiral}
\end{figure}

\begin{proof}
Let $\theta\in\left[0, 1\right]$ and $N\ge 3$. From Dirichlet's Theorem in Diophantine approximation, there exists a rational $p/k\in \left[0,1\right]$ with $1\le k\le N$ such that $\left|\theta-p/k\right|\le 1/(kN)$. Therefore, $\left|\exp\left(2i\pi\theta\right)-\exp\left(2i\pi (p/k)\right)\right|\le 2\pi/(kN)$. Setting $n=k(k+1)/2+p$, one thus gets that 
\begin{align*}
\left|\sqrt{n}\cdot\exp\left(2i\pi\theta\right)-\sqrt{n}\cdot\exp\left(2i\pi\frac{p}{k}\right)\right|\;\le\; \sqrt{\left(\frac{k(k+1)}{2}+p\right)}\cdot\frac{2\pi}{kN}\;\le\;  \frac{4\pi}{N}\cdot
\end{align*}
The visibility claim then follows upon noticing that $\sqrt{n}\le N$ when $N\ge 3$.\\

The claim that  the set of visible points is non--empty is implied by the existence of a vacant strip of the form $0<y<\delta$ for some $\delta>0$~: this is saying that the set of points $(x,y)\in\R^2$ such that their second coordinate satisfies this contraint does not contain any point of the spiral. To see this, note that the distance between a point $\sqrt{n}\cdot\exp\left(2i\pi\frac{p}{k}\right)$ in the spiral set (identified with the planar point $\sqrt{n}\left(\cos\left(2\pi p/k\right), \, \sin\left(2\pi p/k\right)\right)$) and the axis $\left\{y=0\right\}$ is $\sqrt{n}\left|\sin\left(2\pi p/k\right)\right|$. This quantity is non--zero if and only if $p\not\in\left\{0, k\right\}$, in which case $$\sqrt{n}\left|\sin\left(2\pi p/k\right)\right|\;\ge\; \frac{\sqrt{n}}{k}\;\underset{\eqref{defn}}{\gg}\; 1.$$ This concludes the proof.

\end{proof}

The rest of this section is devoted to the proof of the remaining statements in Theorem~\ref{main}.

\begin{proof}
\textbf{Proof of Point~\ref{pt1}.} By definition, the spiral $\mathfrak{S}_d\left(\bm{U}\right)=\left\{\sqrt[d+1]{n}\cdot\bm{u}_n\right\}_{n\ge 1}$ is an orchard with visibility $V$ if and only if for any unit vector $\bm{v}\in\Sph^d$ and any real $\epsilon>0$, there exists a real $0<t<V(\epsilon)$ such that for some index $n\ge 1$, $\left\|t\cdot\bm{v}-\sqrt[d+1]{n}\cdot\bm{u}_n\right\|_2\le \epsilon.$  From~\eqref{identity}, $$\left\|t\cdot\bm{v}-\sqrt[d+1]{n}\cdot\bm{u}_n\right\|_2\;\asymp\; \left|t-\sqrt[d+1]{n}\right|+\sqrt{t\sqrt[d+1]{n}}\cdot\dist_{\Sph^d}\left(\bm{u}_n, \bm{v}\right)\;\ll\; \epsilon.$$ 

Assuming first that $\mathfrak{S}_d\left(\bm{U}\right)$ is an orchard, one infers from this identity in the same way as in the above proof of the lower bound~\eqref{boundvisibility}, on the one hand that  for $\epsilon>0$ small enough, $\sqrt[d+1]{n}\asymp t \ll V(\epsilon)$ and, on the other, that $\dist_{\Sph^d}\left(\bm{u}_n, \bm{v}\right)\ll \epsilon/\sqrt[d+1]{n}.$ Conversely, assuming that one can find an index $n\ll V(\epsilon)^{d+1}$ such that  $\dist_{\Sph^d}\left(\bm{u}_n, \bm{v}\right)\ll \epsilon/\sqrt[d+1]{n}$, setting $t=\sqrt[d+1]{n}$ proves that $\mathfrak{S}_d\left(\bm{U}\right)$ is an orchard with visibility a function $W$ as in~\eqref{defW}. \\

\textbf{Proof that 2(a) implies 2(b).} Assume that the spiral $\mathfrak{S}_d\left(\bm{U}\right)$ is a uniform orchard with visibility $W$ as defined in~\eqref{defW} for some constants $\kappa_{\bm{U}}>0$ and $c_{\bm{U}}\ge 1$. Fix $\epsilon>0$ and an integer 
\begin{equation}\label{hV}
h\ge W(\epsilon).
\end{equation}

Given $\delta>0$, one can cover the unit sphere $\Sph^d$ with a number $N_d$ of spherical caps of radius $\delta$ satisfying the estimate $N_d\asymp \delta^{-d}$ (\footnote{To see this, consider a maximal packing of $\Sph^d$ with spherical caps of radius $\delta/2$. This gives a covering by balls with radius $\delta$. If there are $N_d$ elements in this packing, then elementary volume considerations imply the relation $N_d\delta^d\asymp 1$.}). Denote by $\bm{\Sigma}=\left\{\bm{\sigma}_k\; :\; 1\le k\le C\cdot h^d\cdot W(\epsilon)\right\}$ the centres of spherical caps with  radius $\delta=(h\cdot\sqrt[d]{W(\epsilon)})^{-1}$ covering $\Sph^d$, where $C>0$ is some constant. 

Fix an integer $k$ such that $1\le k\le C\cdot h^d\cdot W(\epsilon)$. Since $\mathfrak{S}_d\left(\bm{U}\right)$ is a uniform orchard, there exists a real $t\in \left[h, \, h+W\left(\epsilon\right)\right]$  and an integer $n\ge 1$ such that 

\begin{align*}
\left\|\sqrt[d+1]{n}\cdot\bm{u}_n-t\cdot\bm{\sigma}_k\right\|_2
\le\;\epsilon.
\end{align*}

Decompose $t$ as $t=\sqrt[d+1]{h^{d+1}+s}$ for some real $s\in\left[0, \, \left(h+W\left(\epsilon\right)\right)^{d+1}-h^{d+1}\right]$ and the integer $n$ as $n=h^{d+1}+s+p$ for some real $p$. From~\eqref{identity}, the above implies that 

\begin{align}\label{ineqdist}
\left|\sqrt[d+1]{h^{d+1}+s+p} - \sqrt[d+1]{h^{d+1}+s}\right|+\sqrt[(d+1)/2]{\left(h^{d+1}+s+p\right)\cdot \left(h^{d+1}+s\right)}\cdot\dist_{\Sph^d}\left(\bm{u}_n, \bm{\sigma}_k\right) \; \ll\;\epsilon.
\end{align}

An elementary Taylor expansion argument shows that, provided that the pa\-ra\-me\-ter $\eta>0$ is small enough, \sloppy the inequality $\left|\sqrt[d+1]{1+x}-1\right|<\eta$ implies that $x\in\left(-2\eta (d+1), 2\eta (d+1)\right)$. Therefore, one infers from~\eqref{ineqdist} that, for small enough $\epsilon$,  $$\frac{\left|p\right|}{h^{d+1}+s}\;\ll\; \frac{\epsilon}{\sqrt[d+1]{h^{d+1}+s}}, \quad i.e.\quad \left|p\right|\;\ll\; \epsilon\cdot\left(h^{d+1}+s\right)^{d/(d+1)}\;\underset{\eqref{hV}}{\ll}\; \epsilon\cdot h^d.$$ 

This is saying that, given an integer $k$ such that $1\le k\le C\cdot h^d\cdot W\left(\epsilon\right)$, an integer $n$ can be chosen such that $0\le n-h^{d+1} \ll h^d\cdot W\left(\epsilon\right)$ and  $\dist_{\Sph^d}\left(\bm{u}_n, \bm{\sigma}_k\right)  \ll\epsilon/h$ (the latter relation follows from~\eqref{ineqdist}). Therefore, since the spherical caps with radius $\delta=(h\cdot\sqrt[d]{W(\epsilon)})^{-1}$  centered at the elements of $\bm{\Sigma}$ cover $\Sph^d$, and since $\delta\ll \epsilon/h$ from~\eqref{boundvisibility}, there exists a constant $K>0$ such that the spherical caps with radius a constant multiple of $\epsilon/h$ centered at the elements of the set $A_{\bm{U}}\left(h, K\cdot h^d\cdot W\left(\epsilon\right) \right)$ cover $\Sph^d$. This concludes the proof.\\

\textbf{Proof that 2(b) implies 2(a).} Let $\bm{v}\in\Sph^d$, $\epsilon>0$ and $h\ge 0$. The goal is to show that there exists an index $n\ge 1$ such that the inequality $\left\|\sqrt[n+1]{n}\cdot\bm{u_n}-t\cdot\bm{v}\right\|_2\le \kappa^{(1)}_{\bm{U}}\cdot\epsilon$ holds  for some real $t\in \left[h, h+c'_{\bm{U}}\cdot V\left(\kappa^{(2)}_{\bm{U}}\cdot\epsilon\right)\right]$, where the constant $\kappa^{(1)}_{\bm{U}}, \kappa^{(2)}_{\bm{U}}>0$ and $c'_{\bm{U}}\ge 1$  depend only on the spherical sequence $\bm{U}=\left\{\bm{u}_n\right\}_{n\ge 1}$. 

Assume first that inequality~\eqref{hV} holds, where $W$ is defined as in~\eqref{defW} with the constants $c_{\bm{U}}\ge 1$ and $\kappa_{\bm{U}}>0$ given in the statement of~2(b). By assumption, there exists a constant $K>0$ and an integer $i$ such that $1\le i\le K\cdot h^d\cdot W\left(\epsilon\right)$ and $\dist_{\Sph^d}\left(\bm{u}_{h^{d+1}+i},\, \bm{v}\right)\ll \epsilon/h$. Set $$n=h^{d+1}+i \quad \textrm{and} \quad t=\sqrt[d+1]{n}.$$ From assumption~\eqref{hV} and from the inequality $\left|\sqrt[d+1]{1+x}-1\right|\le 2\left|x\right|$ valid whenever the left--hand side is well--defined, it follows that $0\le t-h\ll W\left(\epsilon\right)$. Furthermore, it holds that 
\begin{align*}
\left\|\sqrt[n+1]{n}\cdot\bm{u_n}-t\cdot\bm{v}\right\|_2\;=\; \sqrt[d+1]{n}\cdot\dist_{\Sph^d}\left(\bm{u}_n, \bm{v}\right)\;\ll\;\sqrt[d+1]{h^{d+1}+i}\cdot\frac{\epsilon}{h}\;\underset{\eqref{hV}}{\ll}\; \epsilon.
\end{align*}
This completes the proof under assumption~\eqref{hV}. If $h\le W\left(\epsilon\right)$, the previous case implies the existence of an index $n\ge 1$ and of a real $t$ such that $W\left(\epsilon\right) \,\le t\le \, W\left(\epsilon\right) + c'_{\bm{U}}\cdot V\left(\kappa^{(2)}_{\bm{U}}\cdot\epsilon\right)$  and $\left\|\sqrt[n+1]{n}\cdot\bm{u_n}-t\cdot\bm{v}\right\|_2\le \kappa^{(1)}_{\bm{U}}\cdot\epsilon$ for suitable values of the various constants.  As $V$ is monotonic as $\epsilon \rightarrow 0^+$, $$W\left(\epsilon\right) + c'_{\bm{U}}\cdot V\left(\kappa^{(2)}_{\bm{U}}\cdot\epsilon\right)\;\le\; \tilde{c}_{\bm{U}}\cdot V\left(\tilde{\kappa}_{\bm{U}}\cdot\epsilon\right)\quad \textrm{for some} \quad \tilde{\kappa}_{\bm{U}}>0\quad \textrm{and}\quad \tilde{c}_{\bm{U}}\ge 1.$$
This is enough to conclude that the set $\mathfrak{S}_d\left(\bm{U}\right)$ is a uniform orchard with visibility $\epsilon\mapsto \tilde{c}_{\bm{U}}\cdot V\left(\tilde{\kappa}'_{\bm{U}}\cdot\epsilon\right)$ for some $\tilde{\kappa}'_{\bm{U}}>0$. This establishes~2(a).\\

\textbf{Proof of Point 2(i)}. Assume that $R_d\left(\bm{U}, C\right)<\infty$ for some $C>0$. From the definitions of the quantity $R_d\left(\bm{U}, C\right)$ and of the set $A_{\bm{U}}\left(h,x\right)$ (see~\eqref{defAuhx}), one infers that for all $N, h\ge 1$, $$R_d\left(A_{\bm{U}}\left(h,C\cdot N\right)\right)\;\ll\; \frac{1}{\sqrt[d]{N}}.$$ Specialising this relation to the case when, given $h\ge 1$  and $\epsilon>0$, $N$ is the integer part of $\left((h\cdot\epsilon^{-1})/(C)\right)^d$ shows that Point~2(b) holds with $W\left(\epsilon\right)=A\cdot\epsilon^{-d}$ for some $A>0$, whence the claim. \\

\textbf{Proof of Point~2(ii)}. The assumption on the limit implies that the function $V$ in~\eqref{defvisi} is well--defined. It is then immediate that the condition stated in Point~2(b) is satisfied (with $W=V$), whence the claim.\\

\textbf{Proof of Point~3(i)}. A line $(L)$ in $\R^{d+1}$ can be parametrised as the set $\left\{\lambda\bm{v}+t\bm{w}\; :\; t\in\R\right\}$, where $\bm{v}, \bm{w}\in\Sph^d$ are orthogonal and where $\lambda\ge 0$ is the distance from $(L)$ to the origin. Given $\epsilon>0$, a point in $\mathfrak{S}_d\left(\bm{U}\right)\;=\;\left\{\sqrt[d+1]{n}\cdot \bm{u}_n\right\}_{n\ge 1}\subset\R^{d+1}$ lies $\epsilon$--close to this line if and only if there exist an integer $n\ge 1$ and a parameter $t\in\R$ such that $\left\|\sqrt[d+1]{n}\cdot \bm{u}_n - \left(\lambda\bm{v}+t\bm{w}\right)\right\|_2\;\le\;\epsilon.$ From~\eqref{identity}, up to multiplicative constants, this is equivalent to asking that \begin{equation}\label{eqtinterm}\left|\sqrt[d+1]{n}-\sqrt{\lambda^2+t^2}\right|\;\ll\; \epsilon\qquad \textrm{and}\qquad \sqrt{\sqrt[d+1]{n}\cdot\sqrt{\lambda^2+t^2}}\cdot \dist_{\Sph^d}\left(\bm{u}_n, \, \frac{\lambda\bm{v}+t\bm{w}}{\sqrt{\lambda^2+t^2}}\right)\;\ll\;\epsilon.\end{equation} Since $n\ge 1$, the first relation implies that $\sqrt[d+1]{n}\asymp\sqrt{\lambda^2+t^2}$ which, together with the second relation, is easily seen to yield the claimed equivalence.\\

\textbf{Proof of the converse implication in Point~3(ii)}. The goal is to construct a spiral $\mathfrak{S}_d\left(\bm{U}\right)$ which has an empty set of visible points but which is not a uniform orchard. To this end, consider first a uniform orchard $\mathfrak{S}'_d\left(\bm{U}'\right)\;=\;\left\{\sqrt[d+1]{n}\cdot \bm{u}'_n\right\}_{n\ge 1}$ with visibility, say, $V\left(\epsilon\right)\ll \epsilon^{-d}$ (such a point set exists from the comments following the statement of Theorem~\ref{main}). Let $\rho_n=2\cdot V\left(2^{-n}\right)$, in such a way that the quantity $n!-\rho_n$ is positive for any $n$ larger than some integer  $n_0\ge 1$. For $n\ge n_0$, define $\mathcal{A}_n$ as  the annulus  with outer radius $n!$ and inner radius $n!-\rho_n$.\\

Let $\delta>0$ and $\bm{v}_0\in\Sph^d$. Let $\mathcal{D}$ be the intersection between  $\R^{d+1}\backslash \left(\cup_{n\ge n_0}\mathcal{A}_n\right) $ and the $\delta$--neighbourhood of the ray emanating from the origin in direction $\bm{v}_0\in\Sph^d$.  The spiral $\mathfrak{S}_d\left(\bm{U}\right)$ is then defined from the uniform orchard $\mathfrak{S}'_d\left(\bm{U}'\right)$ as follows~: if the index $n\ge 1$ is such that $\sqrt[d+1]{n}\cdot \bm{u}'_n$ lies in $\mathcal{D}$, then set $\bm{u}_n=\bm{u}'_{m_n}$, where $m_n\ge n$ is the smallest index such that $\sqrt[q]{m_n}\cdot\bm{u}'_{m_n}\not\in \mathcal{D}$ (the existence of this index is guaranteed by the uniform orchard property of $\mathfrak{S}'_d\left(\bm{U}'\right)$). Otherwise, set $\bm{u}_n=\bm{u}'_n$. \\

Clearly, $\mathfrak{S}_d\left(\bm{U}\right)$ is not a uniform orchard since it has no point in the region $\mathcal{D}$ which contains arbitrarily long line segments (supported in the direction determined by $\bm{v}_0$). To show that it has an empty set of visible points, consider, given a point $\bm{x}\in\R^{d+1}$ and a direction $\bm{v}\in\Sph^d$,  the half--line $L\left(\bm{x}, \bm{v}\right)$ as defined in~\eqref{defhalfline}. For $n\ge n_0$ large enough, there exist exactly two points $\bm{a}_n, \bm{b}_n\in L\left(\bm{x}, \bm{v}\right)$ with smallest and largest norms, respectively, intersecting the annulus $\mathcal{A}_n$. Denote by $\bm{w}\in\Sph^d$ the direction of the half--line $L'_n$ joining the origin to $\bm{b}_n$ and let $\bm{c}_n$ be the point with minimal norm in $ \mathcal{A}_n$ lying in $L'_n$. Clearly, the largest distance between a point in the intersection $L\left(\bm{x}, \bm{v}\right)\cap \mathcal{A}_n$ and a point in $L'_n\cap \mathcal{A}_n$ is, for $n$ large enough, the quantity $\varepsilon_n=\left\|\bm{a}_n-\bm{c}_n\right\|_2$.\\

Elementary trigonometric considerations then show that $\varepsilon_n$ is at most a constant multiple (depending on $\bm{x}$ and $\bm{v}$) of $\left\|\bm{b}_n-\bm{c}_n\right\|_2/\left\|\bm{b}_n\right\|_2=\rho_n/n!$. Thus, $\lim_{n\rightarrow\infty}\epsilon_n=0$.\\

Since $\mathfrak{S}_d\left(\bm{U}\right)$ coincides with $\mathfrak{S}'_d\left(\bm{U}'\right)$ on $\mathcal{A}_n$ and since $\mathfrak{S}'_d\left(\bm{U}'\right)$ is a uniform orchard, there exists a point in $\mathfrak{S}'_d\left(\bm{U}'\right)$ which is $2^{-n}$ close to $L'_n\cap \mathcal{A}_n$, and therefore $\left(2^{-n}+\varepsilon_n\right)$--close to $L\left(\bm{x}, \bm{v}\right)\cap \mathcal{A}_n$. Upon letting $n$ tend to infinity, this shows that $\dist_2\left(L\left(\bm{x}, \bm{v}\right),  \mathfrak{S}_d \left(\bm{U}\right)\right)=0$, which concludes the proof.\\

\textbf{Proof of Point~\ref{pt4}.} Let $\epsilon>0$. A line segment $(\Delta)$ with length $V\left(\epsilon\right)$ in $\R^{d+1}$ can be parametrised as the set $\left\{\lambda\bm{v}+t\bm{w}\; :\; t_0\le t\le t_0+V\left(\epsilon\right)\right\}$, where $\bm{v}, \bm{w}\in\Sph^d$ are orthogonal and where $\lambda\ge 0$ and $t_0$ are reals. A point in $\mathfrak{S}_d\left(\bm{U}\right)\;=\;\left\{\sqrt[d+1]{n}\cdot \bm{u}_n\right\}_{n\ge 1}\subset\R^{d+1}$ lies $\epsilon$--close to this line segment if and only if there exist an integer $n\ge 1$ and a parameter $t\in\left[t_0,\, t_0+V\left(\epsilon\right)\right]$ such that~\eqref{eqtinterm} holds. The argument is then concluded in the same way as in the proof of Point~3(i).\\

\end{proof}

%
%

\end{document}